\documentclass[12pt]{article}
\usepackage{latexsym,amssymb,upref,amsmath,amsthm,amsfonts}

\newcommand{\achr}{\mathrm{achr}}
\newcommand{\cM}{\mathcal{M}}

\newcommand\cL{\mathcal{L}}

\newtheorem{theorem}{Theorem}
\newtheorem{lemma}[theorem]{Lemma}
\newtheorem{proposition}[theorem]{Proposition}

\newtheorem{conjecture}{Conjecture}
\newtheorem{corollary}[theorem]{Corollary}

\title{On the achromatic number of the Cartesian product of two complete graphs}

\author{Mirko Hor\v n\' ak\\
Institute of Mathematics, P.J. \v{S}af\'arik University\\
Jesenn\'a 5, 040\ 01 Ko\v{s}ice, Slovakia\\
E-mail address: mirko.hornak$@$upjs.sk}
\date{}

\begin{document}
\maketitle

\begin{abstract}
A vertex colouring $f:V(G)\to C$ of a graph $G$ is complete if for any $c_1,c_2\in C$ with $c_1\ne c_2$ there are in $G$ adjacent vertices $v_1,v_2$ such that $f(v_1)=c_1$ and $f(v_2)=c_2$. The achromatic number of $G$ is the maximum number $\achr(G)$ of colours in a proper complete vertex colouring of $G$. Let $G_1\square G_2$ denote the Cartesian product of graphs $G_1$ and $G_2$. In the paper $\achr(K_{r^2+r+1}\square K_q)$ is determined for an infinite number of $q$s provided that $r$ is a finite projective plane order.
\end{abstract}


\section{Introduction}

Consider a finite simple graph $G$ and a finite set of colours $C$. A vertex colouring $f:V(G)\to C$ is \textit{complete} if for any pair $c_1,c_2$ of distinct colours in $C$ one can find a pair $v_1,v_2$ of adjacent vertices in $G$ such that $f(v_i)=c_i$, $i=1,2$. Obviously, if $f$ is proper (adjacent vertices receive distinct colours) and $|C|$ is minimum possible (\textit{i.e.}, $|C|=\chi(G)$, the chromatic number of $G$), then $f$ is necessarily complete.

The \textit{achromatic number} of $G$, in symbols $\achr(G)$, is the \textit{maximum} number of colours in a proper complete  vertex colouring of $G$. This invariant was introduced by Harary, Hedetniemi, and Prins in~\cite{HaHePr}. The problem of determining the achromatic number is NP-complete even for trees, see Cairnie and Edwards~\cite{CaE}. So, it is not surprising that exact results concerning the achromatic number are quite rare. A comprehensive bibliography for the achromatic number is maintained by Edwards~\cite{E}.

Some attention was paid to the achromatic number of graphs created by graph operations. Hell and Miller in \cite{HeM} analysed $\achr(G_1\times G_2)$ where $G_1\times G_2$ is the categorical product of graphs $G_1$ and $G_2$ (in the paper we use the notation taken from the monograph Imrich and Klav\v zar~\cite{IK}). 

The Cartesian product $G_1\square G_2$ of graphs $G_1$ and $G_2$ is the graph with $V(G_1\square G_2)=V(G_1)\times V(G_2)$, in which $(x_1,y_1)$ is joined by an edge to $(x_2,y_2)$ if and only if either $x_1=x_2$ and $\{y_1,y_2\}\in E(G_2)$ or $\{x_1,x_2\}\in E(G_1)$ and $y_1=y_2$. 

A motivation for the study of $\achr(K_p\square K_q)$ comes from the observation by Chiang and Fu~\cite{ChiF1} stating that the assumption $\achr(G_1)=p$ and $\achr(G_2)=q$ implies $\achr(G_1\square G_2)\ge\achr(K_p\square K_q)$. 

Evidently, since the graphs $K_p\square K_q$ and $K_q\square K_p$ are isomorphic to each other, when looking for $\achr(K_p\square K_q)$ we may suppose without loss of generality that $p\le q$.

Let $p,q$ be integers. In the paper we work with \textit{integer intervals} that are denoted as follows:
\begin{equation*}
[p,q]=\{z\in\mathbb Z:p\le z\le q\},\qquad[p,\infty)=\{z\in\mathbb Z:p\le z\}.
\end{equation*}
For a finite set $A$ and $k\in[0,|A|]$ the set $\binom Ak$ consists of all $k$-element subsets of $A$. 

The element in the $i$th row and the $j$th column of a matrix $M$ is presented as $(M)_{i,j}$. The submatrix of $M$ corresponding to the $i$th row of $M$ is denoted by $R_i(M)$.

The value of $\achr(K_p\square K_q)$ is known for each pair $(p,q)$ satisfying $1\le p\le6$ and $p\le q$. Besides the trivial case $p=1$ ($K_1\square K_q$ is isomorphic to $K_q$, hence $\achr(K_1\square K_q)=q$), the case $p\in\{2,3,4\}$ was settled by Hor\v n\'ak and Puntig\'an~\cite{HoPu} (for $p\in\{2,3\}$ the result was rediscovered in~\cite{ChiF1}), the case $p=5$ by Hor\v n\'ak and P\v cola~\cite{HoPc1}, \cite{HoPc2}, and the case $p=6$ by Hor\v n\'ak~\cite{Ho3}, \cite{Ho1}, \cite{Ho2}. The achromatic number of $K_p\square K_q$, where $r$ is an odd order of a finite projective plane, was determined in Chiang and Fu~\cite{ChiF2} (for $r=3$ see already Bouchet~\cite{B}). Some values of $\achr(K_p\square K_q)$ with $p\le6$ will be used in Section~\ref{asan}. They are summarised here:

\begin{theorem}\label{known}
$1.$ If $q\in[3,\infty)$, then $\achr(K_2\square K_q)=q+1$.

$2.$ If $q\in[4,\infty)$, then $\achr(K_3\square K_q)=\left\lfloor\frac{3q}2\right\rfloor$.	

$3.$ If $q\in[25,\infty)$, then $\achr(K_4\square K_q)=\left\lfloor\frac{5q}3\right\rfloor$.

$4.$ If $q\in[43,\infty)$, then $\achr(K_5\square K_q)=\left\lfloor\frac{9q}5\right\rfloor$.		

$5.$ If $q\in[41,\infty)$ and $q\equiv1\pmod2$, then $\achr(K_6\square K_q)=2q+3$.	

$6.$ If $q\in[42,\infty)$ and $q\equiv0\pmod2$, then $\achr(K_6\square K_q)=2q+4$.	
\end{theorem}

What follows is (up to the notation) a natural and somehow standard (\textit{cf.}~\cite{HoPu}) approach to dealing with a proper complete vertex colouring of the Cartesian product of two complete graphs.

Suppose that $p,q\in[1,\infty)$, $V(K_s)=[1,s]$ for $s\in\{p\}\cup\{q\}$, $C$ is a finite set and $f:V(K_p\square K_q)\to C$ is a proper complete vertex colouring. Let $M(f)$ be the $p\times q$ matrix with $(M(f))_{i,j}=f(i,j)$. The fact that $f$ is proper means that each row of $M(f)$ consists of $q$ distinct colours of $C$, and similarly each column of $M(f)$ consists of $p$ distinct colours of $C$. Because of the completeness of $f$ for any $\{c_1,c_2\}\in\binom C2$ there is a line (a row or a column) of $M(f)$ that contains both $c_1$ and $c_2$. Let $\cM(p,q,C)$ be the set of all $p\times q$ matrices $M$ with elements from $C$ such that $M$ has all above properties of the matrix $M(f)$.

Conversely, let $M\in\cM(p,q,C)$. It is obvious to see that the colouring $f_M:V(K_p\square K_q)\to C$ defined by $f_M(i,j)=(M)_{i,j}$ is a proper complete vertex colouring of the graph $K_p\square K_q$. Thus, we have just proved 

\begin{proposition}\label{mat}
If $p,q\in[1,\infty)$ and $C$ is a finite set, then the following statements are equivalent:

$1.$ There is a proper complete vertex colouring of the graph $K_p\square K_q$ using as colours elements of $C$.

$2.$ $\cM(p,q,C)\ne\emptyset$.
\end{proposition}

We shall need a subset $\cM^*(p,q,C)$ of $\cM(p,q,C)$ consisting of matrices $M$, which satisfy the additional condition that for any $\{c_1,c_2\}\in\binom C2$ there is a row (not merely a line) of $M$ with both $c_1,c_2$.

Let $r\in[2,\infty)$. A \textit{finite projective plane of order $r$} is a pair $(P,\cL)$, where $P$ is a finite set of elements called \textit{points}, and $\cL$ is a set of subsets of $P$ called \textit{lines}, such that the following axioms are fulfilled:
\vskip1mm

$A_1.$ If $p_1,p_2\in P$, $p_1\ne p_2$, there is exactly one line $L(p_1,p_2)\in\cL$ such that $\{p_1,p_2\}\subseteq L(p_1,p_2)$.

$A_2.$ If $L_1,L_2\in\cL$, $L_1\ne L_2$, then $P_1\cap P_2\ne\emptyset$.

$A_3.$ There are four distinct points $\tilde p_1,\tilde p_2,\tilde p_3,\tilde p_4\in P$ determining six distinct lines, \textit{i.e.,} $|\bigcup_{\{i,j\}\in\binom{[1,4]}{2}}\{L(\tilde p_i,\tilde p_j)\}|=6$.

$A_4.$ There is $\tilde L\in\cL$ such that $|\tilde L|=r+1$.
\vskip1mm

\noindent It is well known that points and lines of a finite projective plane $(P,\cL)$ of order $r$ have the following basic properties:
\vskip1mm

$B_1.$  If $L_1,L_2\in\cL$, $L_1\ne L_2$, then $|L_1\cap L_2|=1$.

$B_2.$ If $L\in\cL$, then $|L|=r+1$.

$B_3.$ If $p\in P$, then $|\{L\in\cL:p\in L\}|=r+1$.

$B_4.$ $|P|=r^2+r+1$.

$B_5.$ $|\cL|=r^2+r+1$.
\vskip1mm

Given $r\in[2,\infty)$, to determine whether there exists a finite projective plane of order $r$ (\textit{i.e.}, whether $r$ is a \textit{finite projective plane order}), is in general a notoriously hard problem of finite combinatorics. All positive results available so far are restricted to $r=q^e$, where $q$ is a prime number and $e\in[1,\infty)$.

\section{Some auxiliary results}

The following lemma is well known, \textit{cf.}~\cite{HoPu}. We include its proof here for a better readability of the paper.

\begin{lemma}\label{opt}
If $p\in[1,\infty)$, $q\in[p,\infty)$, $C$ is a set of size $a=\achr(K_p\square K_q)$, $M\in\cM(p,q,C)$ and $l$ is the smallest of frequencies of elements in $M$, then the following hold:

$1.$ $l\le p$;

$2.$ $l\le\left\lfloor\frac{pq}{a}\right\rfloor$;

$3.$ $a\le l(p+q-l-1)+1$.
\end{lemma}

\begin{proof}
1. The vertex colouring $f_M$ is proper, hence each element of $C$ appears in any row of $M$ at most once.

2. The set of colour classes of $f_M$ is a partition of the set $V(K_p\square K_q)$ of cardinality $pq$ so that $pq\le al$ and $l\le\left\lfloor\frac{pq}{a}\right\rfloor$.

3. Let $\gamma\in C$ be a colour of $f_M$ of frequency $l$. A matrix created from $M$ by permuting the rows and the columns of $M$ evidently belongs to $\cM(p,q,C)$. Therefore, we may suppose without loss of generality that $(M)_{i,i}=\gamma$ for every $i\in[1,l]$. The completeness of $f_M$ means that the neighbourhood $N$ of the colour class $\{(i,i):i\in[1,l]\}$ corresponding to $\gamma$ contains a vertex of each colour in $C\setminus\{\gamma\}$. Since $|N|=ql+(p-l)l-l\le|C|-1$, we have $a=|C|\le l(p+q-l-1)+1$.
\end{proof}

\begin{lemma}\label{+1}
	If $p\in[3,\infty)$, $q\in[2p-1,\infty)$, $a=\achr(K_p\square K_q)$, $C$ is an $a$-element colour set, $\cM^*(p,q,C)\ne\emptyset$, and $d\notin C$ for a colour $d$, then $\cM^*(p,q+1,C\cup\{d\})\ne\emptyset$ and
	$\achr(K_p\square K_{q+1})\ge\achr(K_p\square K_q)+1$.
\end{lemma}

\begin{proof}
For $M\in\cM(p,q,C)$ we consider the block matrix $M^+=(MJ_p(d))$, in which $J_p(d)$ is the $p\times 1$ matrix with all elements equal to $d$.

For each $i\in[1,p]$ we define recurrently an $i\times(q+1)$ matrix $M_i^+$. First, we put $M_1^+=M^+$. For $i\in[2,p]$, if the matrix $M_{i-1}^+$ is already defined, we construct a matrix $M_i^+$ from the block matrix $M_i=\begin{pmatrix}
M_{i-1}^+\\R_i(M^+)\end{pmatrix}$ by interchanging elements $(M_i)_{i,q+1}=d$ and $(M_i)_{i,j}$ for a suitable $j\in[1,q]$ in such a way that lines of $M_i^+$ contain pairwise distinct elements (so that $f_{M_i^+}$ is a proper vertex colouring of $K_i\square K_{q+1}$).

To see that this is doable realise that there are two reasons why an integer from $[1,q]$ cannot be chosen as $j$. The first one is that the $j$th column of $M_{i-1}^+$ contains $d$, and the second one is that $(M_i)_{i,j}$ is an element of the $(q+1)$th column of $M_{i-1}^+$. Therefore, the total number of integers from $[1,q]$, which are not a valid choice for $j$, is $2(i-1)$, and $M_i^+$ can be created in a required way, since $q-2(i-1)\ge q-2(p-1)=q+2-2p\ge1$.

Thus $f_{M_p^+}$ is a proper vertex colouring of $K_p\square K_{q+1}$. As $M\in\cM^*(p,q,C)$, having in mind the construction of $M_p^+$ and the fact that the colour $d$ is present in all $p$ rows of the matrix $M_p^+$, it is clear that the colouring $f_{M_p^+}$ is complete, too. Therefore, $M_p^+\in\cM^*(p,q+1,C\cup\{d\})\ne\emptyset$ and $\achr(K_p\square K_{q+1})\ge|C\cup\{d\}|=a+1=\achr(K_p\square K_q)+1$. 
\end{proof}

\begin{lemma}\label{fppl}
If $r$ is a finite projective plane order and $s\in[r+1,\infty)$, then

$1.$ there exists a colour set $C$ of size $(r^2+r+1)s$ such that $\cM^*(r^2+r+1,$ $(r+1)s,C)\ne\emptyset$; 

$2.$ $\achr(K_{r^2+r+1}\square K_{(r+1)s})\ge(r^2+r+1)s$.
\end{lemma}

\begin{proof}
1. Let $(P,\cL)$ be a finite projective plane of order $r$ with $P=\{p_k:k\in[1,r^2+r+1]\}$ and $\cL=\{L_k:k\in[1,r^2+r+1]\}$ (see the properties $B_4$, $B_5$). Consider an $(r^2+r+1)\times(r+1)$ matrix $M$ with elements from $P$ such that, for each $i\in[1,r^2+r+1]$, $L_i$ is equal to the set $\{(M)_{i,j}:j\in[1,r+1]\}$ of elements in the $i$th row of $M$.

Given $k\in[1,r^2+r+1]$ and $l\in[1,r+1]$, replace the $l$th copy of $p_k$ in $M$ with $p_k^l$ (by $B_3$, $p_k$ appears in $r+1$ distinct lines of $\cL$, and so in $r+1$ distinct rows of $M$); we suppose that the ordering of copies of $p_k$ in $M$ is ``inherited'' from the lexicographical ordering of pairs $(i,j)\in[1,r^2+r+1]\times[1,r+1]$ with $(M)_{i,j}=p_k$. Denote by $M'$ the $(r^2+r+1)\times(r+1)$ matrix obtained from $M$ if each point of $P$ in $M$ is replaced in the above way with $p_k^l$, where $(k,l)\in[1,r^2+r+1]\times[1,r+1]$.

For $z\in\mathbb Z$ let$(z)_s$ be the unique $t\in[1,s]$ satisfying $t\equiv z\pmod s$. Further, let $M_k^s$ be the $(r+1)\times s$ matrix with elements from $\{p_k\}\times[1,s]$ defined by $(M_k^s)_{i,j}=(p_k,(i+j-1)_s)$, \textit{i.e.},
\[
M_k^s=
\begin{pmatrix}
(p_k,1) &(p_k,2) &\dots &(p_k,s-1) &(p_k,s)\\
(p_k,2) &(p_k,3) &\dots &(p_k,s) &(p_k,1)\\
\vdots &\vdots &\ddots &\vdots &\vdots\\
(p_k,r) &(p_k,r+1) &\dots &(p_k,r-2) &(p_k,r-1)\\
(p_k,r+1) &(p_k,r+2) &\dots &(p_k,r-1) &(p_k,r)
\end{pmatrix}
\]

Finally, let $M_s$ be the $(r^2+r+1)\times(r+1)s$ matrix obtained from $M'$ if each $p_k^l$ with $(k,l)\in[1,r^2+r+1]\times[1,r+1]$ is replaced with the $1\times s$ block matrix equal to the $l$th row submatrix of $M_k^s$.

Let us show that $M_s\in\cM^*(r^2+r+1,(r+1)s,C)$ where the set of colours $C=\{(p_k,t):k\in[1,r^2+r+1],t\in[1,s]\}$ is of size $(r^2+r+1)s$. First, the $i$th row of $M_s$, $i\in[1,r^2+r+1]$, consists of $(r+1)s$ distinct elements of $C$ (corresponding to $r+1$ points of $L_i$, see $B_2$). Next, the assumption $s\ge r$ guarantees that each column of $M_s$ consists of $r^2+r+1$ distinct elements of $C$ (even if there is a column of $M'$ containing, for some $k\in[1,r^2+r+1]$, all $p_k^l$ with $l\in[1,r+1]$). So, $M_s$ represents a proper vertex colouring.

If $k,l\in[1,r^2+r+1]$, $k\ne l$, by the axiom $A_1$ there is a unique $i\in[1,r^2+r+1]$ such that $\{p_k,p_l\}\subseteq L_i$. Therefore, for any $t,u\in[1,s]$, both $(p_k,t)$ and $(p_l,u)$ belong to the $i$th row of $M_s$. If $k\in[1,r^2+r+1]$ and $t,u\in[1,s]$, $t\ne u$, both $(p_k,t)$ and $(p_k,u)$ appear in $r+1$ rows of $M_s$ (corresponding to $r+1$ rows of $M$ containing $p_k$). Thus $M_s$ represents a complete vertex colouring, too; moreover, $M_s\in\cM^*(r^2+r+1,(r+1)s,C)\ne\emptyset$.

2. Since  $\cM(r^2+r+1,(r+1)s,C)\supseteq\cM^*(r^2+r+1,(r+1)s,C)\ne\emptyset$ (see Lemma~\ref{fppl}.1), using Proposition~\ref{mat} we obtain $\achr(K_{r^2+r+1}\square K_{(r+1)s})\ge|C|=(r^2+r+1)s$.
\end{proof}

Note that the structure of the matrix $M$ from the proof of Lemma~\ref{fppl} (that depends on the projective plane order $r$) is ``largely'' various, but it determines the structure of matrices $M'$ and $M_s$ in a unique way. For example, in the case $r=2$ and $s=3$ we have
\[
M=
\begin{pmatrix}
p_1 &p_2 &p_3\\
p_3 &p_4 &p_5\\
p_5 &p_6 &p_1\\
p_4 &p_1 &p_7\\
p_7 &p_3 &p_6\\
p_5 &p_7 &p_2\\
p_2 &p_6 &p_4
\end{pmatrix}
,\quad
M'=
\begin{pmatrix}
p_1^1 &p_2^1 &p_3^1\\
p_3^2 &p_4^1 &p_5^1\\
p_5^2 &p_6^1 &p_1^2\\
p_4^2 &p_1^3 &p_7^1\\
p_7^2 &p_3^3 &p_6^2\\
p_5^3 &p_7^3 &p_2^2\\
p_2^3 &p_6^3 &p_4^3
\end{pmatrix}
,
\]
and then the matrix $M_3$ is
\begin{equation*}
\begin{pmatrix}
(p_1,1) &(p_1,2) &(p_1,3) &(p_2,1) &(p_2,2) &(p_2,3) &(p_3,1) &(p_3,2) &(p_3,3)\\
(p_3,2) &(p_3,3) &(p_3,1) &(p_4,1) &(p_4,2) &(p_4,3) &(p_5,1) &(p_5,2) &(p_5,3)\\
(p_5,2) &(p_5,3) &(p_5,1) &(p_6,1) &(p_6,2) &(p_6,3) &(p_1,2) &(p_1,3) &(p_1,1)\\
(p_4,2) &(p_4,3) &(p_4,1) &(p_1,3) &(p_1,1) &(p_1,2) &(p_7,1) &(p_7,2) &(p_7,3)\\
(p_7,2) &(p_7,3) &(p_7,1) &(p_3,3) &(p_3,1) &(p_3,2) &(p_6,2) &(p_6,3) &(p_6,1)\\
(p_5,3) &(p_5,1) &(p_5,2) &(p_7,3) &(p_7,1) &(p_7,2) &(p_2,2) &(p_2,3) &(p_2,1)\\
(p_2,3) &(p_2,1) &(p_2,2) &(p_6,3) &(p_6,1) &(p_6,2) &(p_4,3) &(p_4,1) &(p_4,2)
\end{pmatrix}
.
\end{equation*}

\section{Main theorem}

\begin{theorem}\label{main}
If $r$ is a finite projective plane order, $s\in[r^3+1,\infty)$ and $t\in[0,r]$, then 
\[
(r^2+r+1)s+t\le\achr(K_{r^2+r+1}\square K_{(r+1)s+t})\le(r^2+r+1)s+rt.
\]
\end{theorem}

\begin{proof}
Denote $a_{s,t}=\achr(K_{r^2+r+1}\square K_{(r+1)s+t})$ for $t\in[0,r]$. Since $s\ge r^3+1\ge r+1$, from Lemma~\ref{fppl} we know there is a colour set $C_{s,0}^*$ of size $(r^2+r+1)s$ such that $\cM^*(r^2+r+1,(r+1)s,C_{s,0}^*)\ne\emptyset$ and
\begin{equation}\label{lb0}
a_{s,0}\ge(r^2+r+1)s.
\end{equation}

If $t\in[0,r-1]$, it is an easy exercise to prove the inequality $(r+1)s+t\ge2(r^2+r+1)-1$. Therefore, using (\ref{lb0}) and Lemma~\ref{+1}, by induction on $t$ we see that for any $t\in[0,r]$ there exists a colour set $C_{s,t}^*$ of size $(r^2+r+1)s+t$ with $\cM^*(r^2+r+1,(r+1)s+t,C_{s,t}^*)\ne\emptyset$, which implies
\begin{equation}\label{lbq}
a_{s,t}\ge(r^2+r+1)s+t.
\end{equation}

Now let $C_{s,t}$ be an $a_{s,t}$-element set. By Proposition~\ref{mat} there exists a matrix $M_{s,t}\in\cM(r^2+r+1,(r+1)s+t,C_{s,t})$. If $l_{s,t}$ is the minimum frequency of an element of $C_{s,t}$ in $M_{s,t}$, from Lemma~\ref{opt}.2 and (\ref{lbq}) it follows that
\begin{align}
l_{s,t}&\le\left\lfloor\frac{(r^2+r+1)[(r+1)s+t]}{a_{s,t}}\right\rfloor\le\left\lfloor\frac{(r^2+r+1)[(r+1)s+t}{(r^2+r+1)s}\right\rfloor\nonumber\\
&=r+1.\label{r+1}
\end{align}

In the case $l_{s,t}=r+1$ each colour class of the colouring $f_{M_{s,t}}$ is of cardinality at least $r+1$, hence, by (\ref{lbq}), 
\[
|V(K_{r^2+r+1}\square K_{(r+1)s+t})|=(r^2+r+1)[(r+1)s+t]\ge(r+1)a_{s,t},
\]
so that
\begin{equation}\label{ub1}
a_{s,t}\le(r^2+r+1)s+\left\lfloor\frac{(r^2+r+1)t}{r+1}\right\rfloor=(r^2+r+1)s+rt.
\end{equation}

On the other hand, if $l_{s,t}\le r$ (see (\ref{r+1})), Lemma~\ref{opt}.3 yields
\begin{align}
a_{s,t}&\le l_{s,t}[r^2+r+1+(r+1)s+t-l_{s,t}-1]+1\nonumber\\
&\le r[r^2+r+1+(r+1)s+t-r-1]+1\nonumber\\
&=r[r^2+(r+1)s+t]+1\label{ub2}
\end{align}
(since $r\ge2$, the polynomial $x[r^2+r+1+(r+1)s+t-x-1]+1$ in variable $x$ is increasing for $x\le r$), and from (\ref{ub2}) we obtain
\begin{equation}\label{ub3}
a_{s,t}\le r^3+r(r+1)s+rt+1\le s+r(r+1)s+rt=(r^2+r+1)s+rt.
\end{equation}

From (\ref{ub1}) and (\ref{ub3}) we see that 
\begin{equation}\label{ubq}
a_{s,t}\le(r^2+r+1)s+rt
\end{equation}
independently from the value of $l_{s,t}$, and so, by (\ref{lbq}) and (\ref{ubq}),
for any $t\in[0,r]$ (including $t=k$) we have $(r^2+r+1)s+t\le a_{s,t}\le(r^2+r+1)s+rt$.
\end{proof}

\begin{corollary}
If $r$ is a finite projective plane order and $s\in[r^3+1,\infty)$, then 
\[
\achr(K_{r^2+r+1}\square K_{(r+1)s})=(r^2+r+1)s.
\]
\end{corollary}

\begin{proof}
Take $k=0$ in Theorem~\ref{main}.
\end{proof}

\section{Asymptotic analysis}\label{asan}

\begin{theorem}\label{as}
If $r$ is a finite projective plane order, then 
\[
\lim_{q\to\infty}\frac{\achr(K_{r^2+r+1}\square K_q)}{q}=\frac{r^2+r+1}{r+1}.
\]
\end{theorem}
\begin{proof}
Denote for simplicity $a(p,q)=\achr(K_p\square K_q)$. We have $a(p,q)=a\left(p,(r+1)\left\lfloor\frac q{r+1}\right\rfloor)+k\right)$, where $k=q-(r+1)\left\lfloor\frac q{r+1}\right\rfloor\in[0,r]$. If $s=\left\lfloor\frac q{r+1}\right\rfloor\ge r^3+1$, then, by Theorem~\ref{main},
\begin{align}\label{bds}
\frac{(r^2+r+1)\left\lfloor\frac q{r+1}\right\rfloor}{q}&\le\frac{(r^2+r+1)\left\lfloor\frac q{r+1}\right\rfloor+k}{q}\nonumber\\
&\le\frac{a\left(r^2+r+1,(r+1)\left\lfloor\frac q{r+1}\right\rfloor+k\right)}{q}\nonumber\\
&=\frac{a(r^2+r+1,q)}{q}\le\frac{(r^2+r+1)\left\lfloor\frac q{r+1}\right\rfloor+rk}{q}\nonumber\\
&\le\frac{(r^2+r+1)\left\lfloor\frac q{r+1}\right\rfloor+r^2}{q}.
\end{align}
Now, having in mind that $\lim_{q\to\infty}\frac{\left\lfloor\frac q{r+1}\right\rfloor}{q}=\frac1{r+1}$ and $\lim_{q\to\infty}\frac{r^2}{q}=0$, from (\ref{bds}) we obtain $\lim_{q\to\infty}\frac{a(r^2+r+1,q)}{q}=\frac{r^2+r+1}{r+1}$.
\end{proof}

From the known results for $\achr(K_p\square K_q)$ with $p\le6$, see~\cite{HoPu} ($p\le4$), \cite{HoPc1} ($p=5$) and \cite{Ho1}, \cite{Ho2} ($p=6$), we can easily deduce the existence and the value of the limit $l_p=\lim_{q\to\infty}\frac{\achr(K_p\square K_q)}{q}$. Namely, we have $l_1=1=l_2$, $l_3=\frac32$, $l_4=\frac53$, $l_5=\frac95$ and $l_6=2$. These facts together with Theorem~\ref{as} motivate us to formulate

\begin{conjecture}\label{cj}
If $p\in[1,\infty)$, then $\lim_{q\to\infty}\frac{\achr(K_p\square K_q)}{q}$ does exist and is a rational number.
\end{conjecture}
\vskip2mm

\noindent{\Large{\bf Acknowledgement}}
\vskip1mm

\noindent This work was supported by the Slovak Research and Development Agency under the contract APVV-19-0153 and by the grant VEGA 1/0574/21.

\end{document}